\documentclass[12pt, oneside,reqno]{amsart}
\usepackage{amsmath, amsfonts, amssymb}
\usepackage{eucal}
\usepackage{mathrsfs}
\usepackage{latexsym}
\usepackage{cite}
\usepackage{bbm}
\usepackage{todonotes}
\usepackage{mathtools}
\usepackage{graphicx}
\usepackage{relsize}
\usepackage{exscale}
\makeatletter
\newsavebox\myboxA
\newsavebox\myboxB
\newlength\mylenA

\newcommand*\xoverline[2][0.75]{%
    \sbox{\myboxA}{$\m@th#2$}%
    \setbox\myboxB\null% Phantom box
    \ht\myboxB=\ht\myboxA%
    \dp\myboxB=\dp\myboxA%
    \wd\myboxB=#1\wd\myboxA% Scale phantom
    \sbox\myboxB{$\m@th\overline{\copy\myboxB}$}%  Overlined phantom
    \setlength\mylenA{\the\wd\myboxA}%   calc width diff
    \addtolength\mylenA{-\the\wd\myboxB}%
    \ifdim\wd\myboxB<\wd\myboxA%
       \rlap{\hskip 0.5\mylenA\usebox\myboxB}{\usebox\myboxA}%
    \else
        \hskip -0.5\mylenA\rlap{\usebox\myboxA}{\hskip 0.5\mylenA\usebox\myboxB}%
    \fi}
\makeatother

\newcommand{\rn}{\mathbb R^n}

\newcommand{\sn}{S^{n-1}}

\newcommand{\kno}{\mathcal K^n_o}

\newcommand{\sno}{\mathcal S^n_o}

\newcommand{\bu}{\pmb{\nu}}
\newcommand{\balpha}{\pmb{\alpha}}

\newcommand\wtilde[1]{\overset{\lower.4ex\hbox{$\scriptstyle \sim$}}{#1}}
\newcommand\wst[1]{\overset{\lower.5ex\hbox{$\scriptscriptstyle \sim$}}{#1}}
\newcommand{\blb}{\raise.3ex\hbox{$\scriptstyle \pmb \lbrack$}}
\newcommand{\sblb}{\raise.1ex\hbox{$\scriptscriptstyle \pmb \lbrack$}}
\newcommand{\brb}{\raise.3ex\hbox{$\scriptstyle \pmb \rbrack$}}
\newcommand{\sbrb}{\raise.1ex\hbox{$\scriptscriptstyle \pmb \rbrack$}}
\newcommand{\bla}{\raise.2ex\hbox{$\scriptstyle\pmb \langle$}}
\newcommand{\sbla}{\raise.1ex\hbox{$\scriptscriptstyle\pmb \langle$}}
\newcommand{\bra}{\raise.2ex\hbox{$\scriptstyle\pmb \rangle$}}
\newcommand{\sbra}{\raise.1ex\hbox{$\scriptscriptstyle\pmb \rangle$}}
\newcommand{\blrb}{\raise.3ex\hbox{$\scriptstyle \pmb | $}}
\newcommand{\sblrb}{\raise.1ex\hbox{$\scriptscriptstyle \pmb | $}}

\newcommand{\st}{\scriptstyle}
\newcommand{\sst}{\scriptscriptstyle}
\newcommand{\bi}{\item[\raisebox{0pt}{$\st\bullet$}]}
\newcommand{\bii}{\item[\raisebox{0pt}{$\sst\bullet$}]}

\newcommand{\ir}[1]{\item[\rm ({#1})]}
\newcommand{\sir}[1]{\item[\scriptsize\rm ({#1})]}

\newcommand{\wt}{\widetilde}

\newcommand{\psum}{\,{+_{\negthinspace\kern-2pt p}}\,}
\newcommand{\qsum}[1]{\,{+_{\negthinspace\kern-2pt \lower -2pt \hbox{$_{_{#1}}$}}}\,}
\newcommand{\osum}{{+_{\negthinspace\kern-2pt {\rm{o}}}}\,}
\newcommand{\dpsum}{\,{\tilde+_{\negthinspace\kern-1pt p}}\,}
\newcommand{\dqsum}[1]{{\,\wt+_{\negthinspace\kern-1pt #1}}\,}
\newcommand{\lsub}[1]{\hskip -1.5pt\lower.5ex\hbox{$_{#1}$}}

\newcommand{\R}{\mathbb{R}}

\newcommand{\HH}{\mathcal{H}}

\newcommand{\rhok}{\rho\hskip -1pt\lower.4ex\hbox{$_{K}$}}
\newcommand{\nuk}{\nu{\hskip -2pt\lower.2ex\hbox{$_{K}$}}}

\newcommand{\rk}{r{\hskip -2pt\lower.2ex\hbox{$_{K}$}}}
\newcommand{\srk}{r{\hskip -2pt\lower.2ex\hbox{$_{\sst K}$}}}

\title[Dual Curvature Density Equation with Group Symmetry]
{Dual Curvature Density Equation with Group Symmetry}

\author[K. B\"or\"oczky]{K\'aroly J. B\"or\"oczky}
\address{Alfr\'ed R\'enyi Institute of Mathematics,
 Hungarian Academy of Sciences,
 Realtanoda u. 13-15, H-1053,
 Budapest, Hungary}
\email{boroczky.karoly.j@renyi.hu}

\author[\'A. Kov\'acs]{\'Agnes Kov\'acs}
\address{Department of Biostatistics, University of Veterinary Medicine, Budapest}
\email{kovacs.agnes.maria@univet.hu}

\author[S. Mui]{Stephanie Mui}
\address{Department of Mathematics, Georgia Institute of Technology,  686 Cherry St NW, Atlanta, GA 30332, USA}
\email{smui3@gatech.edu}

\author[G. Zhang]{Gaoyong Zhang}
\address{Department of Mathematics,
Courant Institute of Mathematical Sciences,
New York University,
251 Mercer Street,
New York, NY 10012, USA}
\email{gaoyong.zhang@courant.nyu.edu}


\thanks{Research supported, in part, by NKFIH ADVANCED 150613,
NSF Grant  DMS--2005875, NSF Grant  DMS--2402038.}

\begin{document}

\maketitle

\newtheorem{lemma}{Lemma}[section]
\newtheorem{theo}[lemma]{Theorem}
\newtheorem{example}[lemma]{Example}
\newtheorem{defi}[lemma]{Definition}
\newtheorem{claim}[lemma]{Claim}
\newtheorem{coro}[lemma]{Corollary}
\newtheorem{conj}[lemma]{Conjecture}
\newtheorem{prop}[lemma]{Proposition}
\newtheorem{remark}[lemma]{Remark}
\newtheorem{problem}[lemma]{Problem}
\newtheorem{const}[lemma]{Construction}

\begin{abstract}

This paper studies the general $L_p$ dual curvature density equation under a group symmetry assumption.
This geometric partial differential equation arises from the general $L_p$ dual Minkowski problem of prescribing
the $L_p$ dual curvature measure of convex bodies.
It is a Monge-Amp\`ere type equation on the unit sphere. If the density function of the dual curvature measure
is invariant
under a closed subgroup of the orthogonal group, the geometric partial differential equation
is solved in this paper for certain range of negative $p$ using a variational method.
This work generalizes recent results on the $L_p$ dual Minkowski problem of origin-symmetric convex bodies.
\end{abstract}

%\noindent{\bf MSC 2010 } Primary: 35J96, Secondary: 52A40

\section{Introduction}

This paper studies the following Monge-Amp\`ere type equation on the unit sphere $\sn$,
\begin{equation}\label{pde}
\det \big(\nabla^2 h(v)|_{v^\perp}\big) = h(v)^{p-1} F(\nabla h(v))^{n-q} f(v),
\ \ \ v\in \sn,
\end{equation}
where $p, q\in \R$,  $f$ is a nonnegative integrable function on $\sn$,
$F$ is a nonnegative continuous homogenous function of degree 1 in
$\rn$, $n\ge 2$, $\nabla^2h(v)|_{v^\perp}$ is the restriction of the
Hessian to the subspace $v^\perp$, and $\nabla h$ is the gradient of the unknown nonnegative
homogeneous function $h$ of degree 1. Equation \eqref{pde} is called the {\it dual curvature
density equation}. It arises from the general $L_p$ dual Minkowski problem in convex geometry
posed in \cite{LYZ18}, that asks about prescribing the
$L_p$ $q$th-dual curvature measure $\wt C_{p,q}(K,Q;\cdot)$ of a convex body $K$ with respect to
a star body $Q$ in $\rn$ (see Section 2 for the definition).

\begin{problem}[General $L_p$ dual Minkowski problem]\label{LpdualMinkowskiProblem}
Given a finite Borel measure $\mu$ on $\sn$ and a star body $Q$ in $\rn$,
find the necessary and sufficient conditions so that
there exists a convex body $K$ in $\rn$ that solves the geometric measure equation,
\begin{equation}\label{meq}
\wt C_{p,q}(K,Q; \cdot) = \mu.
\end{equation}
\end{problem}

In the absolutely continuous case, that is, the measure $\mu$ has a density function $f$, then the measure equation \eqref{meq} becomes the partial
differential equation \eqref{pde}. The general $L_p$ dual Minkowski problem can then be stated as follows:

\begin{problem}[Continuous $L_p$ dual Minkowski problem]\label{CpqdMP}
Given a nonnegative integrable function $f$ on $\sn$ and a nonnegative continuous homogeneous
function $F$ of degree 1 in $\rn$,
find the necessary and sufficient conditions so that there exists a nonnegative
convex homogeneous function $h$ of degree 1 that solves equation \eqref{pde}.
\end{problem}

The general $L_p$ dual Minkowski problem unifies
several well-known Minkowski problems as special cases.
The first special case of $q=n$
is known as the $L_p$ Minkowski problem introduced by Lutwak \cite{Lut93a} in 1993, which is a major problem
in the $L_p$ Brunn-Minkowski theory in convex geometry. The
$L_p$ Minkowski problem includes three well-known problems: the classical Minkowski
problem when $p=1$, the largely unsolved logarithmic Minkowski problem when $p=0$,
and also the unsolved centro-affine Minkowski problem when $p=-n$. The $L_p$ Minkowski problem
has been studied extensively by many authors, see the survey paper \cite{HYZ25} for more details and references.
The second special case of $q=0$ and $F(x)=|x|$ is known as the $L_p$ Aleksandrov problem, which includes the classical Aleksandrov problem ($p=0$), see \cite{HLYZ18}.
The third special case of $p=1$ and $F(x)=|x|$ is the dual Minkowski problem, first studied in the celebrated
paper \cite{HLYZ16} and then subsequently in a series of papers 
\cite{Zha17, LSW20, ChL18, BLYZ19,ElH23,CLZ19,BoH15,BLYZ13,HeP18,CCL21}.

A number of cases of the $L_p$ dual Minkowski problem have been studied. The case of $p>0$ and $q<0$
is quite similar to the classical Minkowski problem, while the case of
$q>0$ and $p<0$ is more challenging.  We mention the following solved cases of the
equation \eqref{pde}:
\begin{itemize}
\bi $p=0$.
\begin{itemize}
\bii $q=0$. Solved by B\"or\"oczky-LYZ-Zhao \cite{BLYZ20} and Li-Wang \cite{LW18}.
\bii $0<q<n$. Solved for even $f$ when $F(x)=|x|$. See \cite{HLYZ16, BLYZ19, Zha18}
\end{itemize}

\bi $p>0$.
\begin{itemize}
\bii $q=0$. Solved by Huang-LYZ \cite{HLYZ18} when $F(x)=|x|$.

\bii $q\neq 0, p$. Solved when $F(x)=|x|$. See Chen-Li \cite{ChL21}, Lu-Pu \cite{LuP21},
B\"or\"oczky-Fodor \cite{BoF19}, and Huang-Zhao \cite{HuZ18}.
\end{itemize}

\bi $p<0$.
\begin{itemize}
\bii $q=0$. Solved by Huang-LYZ \cite{HLYZ18} when $f$ is even and $F(x)=|x|$, and generalized by Zhao \cite{Zha19} and Mui \cite{Mui22}.

\bii $p\neq q<0$. Solved when $F(x)=|x|$ by Huang-Zhao \cite{HuZ18}.

\bii $-1<p<0$, $q<1+p$, $p\neq q$. Solved by Mui \cite{Mui24} when $f$ is even and $F(x)=|x|$.

\bii $q>n-1$. Solved by Guang-Li-Wang \cite{GLW23} when $F$ is smooth and
$$
p<\left\{
\begin{array}{rcl}
-\frac{(n-1)q}{q-1}&\mbox{if}&q\geq n\\
-\frac{q}{q-n+1}&\mbox{if}&n-1<q< n.
\end{array} \right.
$$
\end{itemize}
\end{itemize}

The purpose of this paper is to prove the existence of a solution for some unsolved cases
of the continuous $L_p$ dual Minkowski problem when $q>0$ and $p<0$. Since the dual curvature
measure $\wt C_{p,q}(K,Q,\cdot)$ is origin-dependent, to be able to solve the problem, sometimes
it is assumed that the bodies $K$ and $Q$ are origin-symmetric and the measure $\mu$ is even.
That is, in the absolutely continuous case, the functions $f$, $F$ and $h$ are assumed to be even. Note that
origin-symmetry and evenness denote the invariance by reflection about the origin, that is an element of
the group $O(n)$. In this paper, we broaden the scope of study by considering more general group symmetry.
This then allows one to solve the equations \eqref{pde} and \eqref{meq} for classes of non-symmetric
bodies and non-even functions.

Let $G$ be a subgroup of $O(n)$. A subset $E\subset \rn$, or a function $f$ over $E$,
is called $G$-invariant if
\[
gE = E,  \ \ g\in G, \ \text{ or } f(gx) = f(x), \ \  g\in G.
\]
A Borel measure $\mu$ on $S^{sn-1}$ is called $G$-invariant if
 \[
 \mu (g E) = \mu(E), \ \ \text{for any } g\in G \text{ and any Borel set } E \subset S^{n-1}.
 \]

For $q>0$, let
\begin{equation}\label{q*}
q^*=\left\{
\begin{array}{rcl}
\frac{q}{q-n+1}&\mbox{if}&q\geq n\\
\frac{(n-1)q}{q-1}&\mbox{if}&1<q< n\\
\infty&\mbox{if}& 0<q\leq 1.
\end{array} \right.
\end{equation}
Note that $q^*>1$ is finite if $q>1$, $(q^*)^*=q$ and $n^*=n$. In addition, $q^*\geq n$ if $q\leq n$.

We will solve the continuous $L_p$ dual Minkowski problem for convex bodies with a group symmetry,
in the range of $q>0$ and $-q^*<p<0$.

\begin{theo}\label{pneqqpos-abs-cont}
Let $q>0$, $-q^*<p<0$, $G$ a closed subgroup of $O(n)$
without a non-zero fixed point, and
$Q$ a $G$-invariant star body in $\rn$.
If $\mu$ is a non-trivial, $G$-invariant, finite Borel
measure on $\sn$ with a density function
$f\in L^s(\sn)$, where $s>1$ if $q\leq 1$, and $s=\frac{1}{1+p/q^*}>1$ if $q>1$,
then there exists a $G$-invariant convex body $K$ in $\rn$ so that it solves
the measure equation $\widetilde{C}_{p,q}(K,Q;\cdot)=\mu$.
\end{theo}

We observe that if a body or a measure is invariant under rotations $g_1, \ldots, g_m$, then
it is invariant under the group $G$ generated by $g_1, \ldots, g_m$.
Thus, classes of convex bodies invariant under subgroups of $O(n)$ are abundant.
We also observe that the solution convex body $K$ must have its centroid at the origin, since $K$ is $G$-invariant, and $G$ has no non-zero fixed point.
An important special case of $G$-invariance is origin-symmetry, which is
the case of $G=\{I, -I\}$ where $I$ is the identity of $O(n)$.
However, there are many more suitable subgroups $G\subset O(n)$,
as discussed in Section~\ref{sect-classes}. For example, take the symmetry
group of a regular simplex. 

The origin-symmetric case of the following statement was studied by
Chen-Chen-Li \cite{CCL21} using the flow method and approximation.
Here, we provide a simple direct proof for a more general statement, using the variational method.

Theorem \ref{pneqqpos-abs-cont}  above implies the following result.

\begin{theo}\label{pde-thm}
Let $q>0$, $-q^*<p<0$, $G$ a closed subgroup of $O(n)$
without a non-zero fixed point, and let
$F$ be a continuous homogeneous function of degree 1 in
$\rn$, with $F(x)>0$ if $x\neq 0$.
If $f$ is a non-vanishing, $G$-invariant, integrable function on $\sn$ with
$f\in L^s(\sn)$, where $s>1$ if $q\leq 1$, and $s=\frac{1}{1+p/q^*}>1$ if $q>1$,
then there exists a nonnegative
convex homogeneous function $h$ of degree 1 with $h(x)>0$ for $x\neq 0$ that solves equation \eqref{pde}.
\end{theo}

No uniqueness of the solution of the dual Minkowski problem holds in general. 
For example, multiple solutions exist in the case $q>2n$ and $f\equiv 1$, 
even if one assumes that the solution $h$ is even (see \cite{CCL21}).
Uniqueness of the solution to the $L_p$ $q$th-dual Minkowski problem \eqref{meq} 
is investigated in \cite{CCL21,LLL22,LiW,IvM23,HuI24}.
Orlicz versions of these Monge-Amp\`ere equations have been
considered by  \cite{LSYY22,FHL22,HLM23,XiZ20,GHWXY19,GHXY20,XYZ22,LiL20}.

Important related variants of the dual Minkowski problem are
the chord Minkowski problem
(see \cite{LXYZ}) and its $L_p$ version (see \cite{XYZZ23}  for $p>0$ and
\cite{YLia,YLib} for $p<0$, also \cite{GXZ} and \cite{XYZZ23}).
Another related problem is the affine dual Minkowski problem, 
proposed by Cai-Leng-Wu-Xi \cite{CLWX}.

The survey article \cite{HYZ25} gives a detailed description of Minkowski problems for
geometric measures with many references.

\section{Preliminaries and Dual Curvature Measures}

For more detailed information on convex geometry, we refer the reader to Gardner \cite{G06book}, Gruber \cite{Gruber07}
and Schneider \cite{Sch14}.

Our setting is {the} Euclidean $n$-space $\R^n$ with $n\geq 2$. Write
 $\langle\cdot,\cdot\rangle$ for the standard inner product and $|\cdot |$ for its induced norm.
 Denote the unit ball  by $B^n=\{x\in\R^n:|x|\leq 1\}$ and the unit sphere by $S^{n-1}=\partial B^n$.
 The notation $\HH^k(\cdot)$ stands for the $k$-dimensional Hausdorff measure
 normalized in a way that it coincides with the Lebesgue measure on $\R^k$,
 and we use the notation $V(\cdot)$ for the $n$-dimensional volume (Lebesgue measure).
 In particular, the volume of the unit ball is $\kappa_n=V(B^n)=\frac{\pi^{\frac n2}}{\Gamma(\frac n2 +1)}$
 and its surface area is  $\HH^{n-1}(S^{n-1})=n\kappa_n$, where $\Gamma$ is Euler's gamma function.

 We call a compact convex set $K\subset\R^n$ with non-empty interior a convex body.
 Denote by $\mathcal{K}^n_{o}$ the family of all convex bodies $K$ that contain the origin in their interior,
 that is, $0\in{\rm int}\,K$.

For a convex compact set $K\subset\R^n$, the support function $h_K : \rn\to \R$ is defined as
$$h_K(y)=\max \{\langle x,y\rangle : x\in K\}$$
where
\begin{equation}
\label{SupportFunctionTranslate}
h_{K-z}(y)=h_K(y)-\langle y,z\rangle
\end{equation}
for any compact convex set $K\subset\R^n$ and $y,z\in\R^n$. Note that the support function
is convex and homogeneous of degree $1$. Furthermore, the support function uniquely defines
a convex body.

A compact set $Q$ in $\rn$ is called star-shaped if
$\lambda x\in Q$ for any $x\in Q$ and $\lambda\in[0,1]$.
For a star-shaped set $Q$ in $\rn$, its radial function
$\rho\lsub{Q}: \mathbb{R}^{n}\backslash\left\{ 0\right\} \rightarrow\mathbb{R}$
is defined by
$$
\rho\lsub{Q}(x)=\max\left\{ \lambda:\text{ }\lambda x\in Q\right\} \text{.}
$$
The radial function is nonnegative and homogeneous of degree $-1$.
We define the family of star bodies $\mathcal S_o^n$ in $\rn$ whose radial functions
are positive and continuous.
Clearly,  $\mathcal{K}^n_{o}\subset \mathcal{S}^n_{o}$.
Every star body is uniquely determined by its radial function.

The support function and the radial function have the following transformation formulas
under a invertible linear transformation $\phi$:
\begin{align}
h_{\phi K}(y) &= h_K(\phi^t y), \label{htrans}\\
\rho\lsub{\phi Q}(x) &=\rho\lsub{K} (\phi^{-1}x). \label{rhotrans}
\end{align}

For a star body $Q\in\mathcal{S}_o^n$, let
\begin{align*}
\|x\|_Q  = \frac1{\rho\lsub{Q}(x)} \mbox{ \ \ for }x\in\R^n \setminus\{0\}.
\end{align*}

For a convex body $K\in\mathcal{K}^n_{o}$, its polar is the convex body
$K^*\in\mathcal{K}^n_{o}$ defined as
$$
K^*=\{x\in\R^n:\,\langle x,y\rangle\leq 1\,\;\forall y\in K\}.
$$
It satisfies $(K^*)^*=K$, $(\lambda K)^*=\lambda^{-1}K^*$ for $\lambda>0$.
If $K\subset C$ for $K,C\in\mathcal{K}^n_{o}$, then $C^*\subset K^*$. In addition,
$$
\rho\lsub{K^*}(y)=\frac1{h_K(y)},  \ \ \ y\in\rn\setminus \{0\}.
$$

For a convex body $K\subset\R^n$, if $v\in S^{n-1}$,
the support hyperplane $H(K,v)$ with unit normal $v$ is
\[
H(K,v) = \{x\in \rn : \langle x,v\rangle=h_K(v) \}.
\]
The intersection $F(K,v) = K\cap H(K, v)$ is the face of $K$
with exterior unit normal $v$.

For $x\in{\partial} K$, let the spherical image of $x$ be defined as
$$
{\pmb\nu}_K(\{x\})=\{v\in S^{n-1}:  x\in H(K,v) \}.
$$
For a Borel set $\eta\subset S^{n-1}$, the reverse spherical image is defined as
$$
{\pmb\nu}_K^{-1}(\eta)=\{x\in{\partial} K:\, {\pmb\nu}_K(x)\cap \eta
\neq \emptyset\}=\bigcup_{v\in\eta}F(K,v).
$$
If $K$ has a unique supporting hyperplane at $x$, then we say that $K$ is smooth at $x$. In this case ${\pmb\nu}_K(\{x\})$ contains exactly one element that we denote by $\nu_K(x)$
and call it the exterior unit normal of $K$ at $x$.

For $K\in \kno$, define  the {\it radial map} of $K$,
\[
r_K : \sn \to \partial K\qquad\text{by}\qquad r_K(u) = \rho_K(u) u \in \partial K, \ \ u\in\sn.
\]

For $\omega\subset\sn$, define the {\it radial Gauss image of $\omega$}  by
\[
\balpha_K(\omega) = \bu_K(r_K(\omega)) \subset \sn,
\]
or equivalently,
\begin{equation}\label{2.2-0}
\balpha_K(\omega) = \{ v\in \sn : \text{$r_K(u) \in H(K,v)$ for some $u\in\omega$} \},
\end{equation}
and thus, for $u\in\sn$,
\begin{equation}\label{2.2-1}
\balpha_K(u) = \{ v\in \sn : r_K(u) \in H(K,v)\}.
\end{equation}

For $K\in \kno$, Aleksandrov-Fenchel-Jessen's surface area measure $S(K,\cdot)$
is a finite Borel measure on $\sn$ defined by
\[
 S(K,\eta)=\HH^{n-1}({\pmb\nu}_K^{-1}(\eta)), \ \ \text{for Borel set } \eta \subset \sn,
\]
and Aleksandrov's integral curvature measure $J(K,\cdot)$ is a finite Borel measure on $\sn$ defined by
\[
J(K,\omega) = \HH^{n-1}(\balpha_K(\omega)), \ \ \text{for Borel set } \omega \subset \sn.
\]

The $L_p$ surface area measure $S_p(K,\cdot)$, $K\in \kno$, is defined by Lutwak \cite{L93a}
\[
dS_p(K,\cdot) = h_K^{1-p} \, dS(K,\cdot).
\]

The $L_p$ integral curvature measure $J_p(K, \cdot)$, $K\in \kno$, is defined by Huang-LYZ \cite{HLYZ18}
\[
dJ_p(K,\cdot) = \rho_K^p\, dJ(K,\cdot).
\]

According to \cite{HLYZ16} and \cite{LYZ18}, if $K\in\mathcal{K}^n_{o}$ and
$\eta\subset S^{n-1}$ is a Borel set,
then the reverse radial Gauss image of $\eta$ is
\begin{align*}
{\pmb\alpha}^*_K(\eta)&=\{u\in S^{n-1}: \rhok(u)u\in F(K,v)\text{ for some }v\in\eta\}\\
&=
\{u\in S^{n-1}: \rhok(u)u\in{\pmb\nu}_K^{-1}(\eta)\},
\end{align*}
which is Lebesgue measurable according to \cite[Lemma~2.2.4]{Sch14}.

For $K, Q\in \sno$ and $q\in\R$, Lutwak's $q$th dual mixed volume $\widetilde{V}_q(K,Q)$ is defined by
\[
\widetilde{V}_q(K,Q)= \frac1n\int_{\sn}\rho_K^{q}(u)\rho^{n-q}_Q(u)\, du.
\]
When $K\in\kno$, the localization (differential) of the $q$th dual mixed volume
is a Borel measure on $S^{n-1}$, $\widetilde{C}_q(K,Q;\cdot)$, called
the $q$th dual curvature measure of $K$ with respect to $Q$ and defined
 by Lutwak-Yang-Zhang  \cite{LYZ18} (by Huang-LYZ \cite{HLYZ16} in the case of $Q=B^n$) as
\begin{equation}\label{dualcurvmeasure}
\widetilde{C}_q(K,Q;\eta)=\frac1n\int_{{\pmb\alpha}^*_K(\eta)}\rho_K^{q}(u)\rho^{n-q}_Q(u)\, du.
\end{equation}

Based on Huang-LYZ \cite{HLYZ16}, LYZ \cite{LYZ18} proved the variational formula for
the $q$th dual mixed volume. Namely,
for $q\neq 0$, $K\in\mathcal{K}_{o}^n$, $Q\in\mathcal{S}_o^n$ and
continuous $\varphi:S^{n-1}\to\R$, if $K_t$ is the corresponding Wulff-shape
$$
K_t=\{x\in\R^n:\,\langle x,v\rangle\leq h_K(v)+t\varphi(v)\,\;\forall v\in S^{n-1}\}
$$
when $|t|$ is small, then
\begin{equation}
\label{dualAlexandrov}
\lim_{t\to 0}\frac{\widetilde{V}_q(K_t,Q)-\widetilde{V}_q(K,Q)}{t}
=q\int_{S^{n-1}}\frac{\varphi(v)}{h_K(v)}\,d \widetilde{C}_q(K,Q;v).
\end{equation}
According to  Lemma~5.1 in \cite{LYZ18},
 if $q\neq 0$ and the Borel function $g:\,S^{n-1}\to \R$ is bounded, then
\begin{equation}
\label{intgCqoin}
\int_{S^{n-1}}g(v)\,d\widetilde{C}_{q}(K,Q;v)
=\frac1n\int_{\partial' K} g(\nu_K(x))\langle \nu_K(x),x\rangle \rho_Q^{n-q}(x)\,d\HH^{n-1}(x).
\end{equation}

The advantage of introducing the star body $Q$ for dual curvature measures
is also apparent in the equi-affine invariance formula
(see Theorem~6.8 in \cite{LYZ18}) stating that
if $\phi\in{\rm SL}(n,\R)$, then
\begin{equation}
\label{Cqaffineinv0}
\int_{S^{n-1}}g(v)\,d\widetilde{C}_{q}(\phi K,\phi Q;v)=
\int_{S^{n-1}} g\left(\frac{\phi^{-t} v}{\|\phi^{-t} v\|}\right)d\widetilde{C}_{q}( K,Q;v),
\end{equation}
{where $\phi^{-t}$ denotes the transpose of the inverse of $\phi$.}

$L_p$ dual curvature measures were also introduced by Lutwak-Yang-Zhang \cite{LYZ18}.
They provide a common framework that unifies several families of geometric measures
in the ($L_p$) Brunn-Minkowski theory and the dual theory: $L_p$ surface area measures,
$L_p$ integral curvature measures, and dual curvature measures, cf. \cite{LYZ18}.
For $p, q\in \R$,  $Q\in\mathcal{S}_{o}^n$, and $K\in\mathcal{K}_{o}^n$,
we define
the $L_p$ $q$th dual curvature measure (or simply $(p,q)$-dual curvature measure)
$\widetilde{C}_{p,q}(K,Q,\cdot)$ of $K$ with respect to $Q$  by the formula
\begin{equation}\label{lpdualcurvmeasureQ}
d\widetilde{C}_{p,q}(K,Q,\cdot)=h_K^{-p}d\widetilde{C}_q(K,Q,\cdot).
\end{equation}
We list three sub-families of $L_p$ dual curvature measures of $K\in\mathcal{K}_{o}^n$:
\begin{itemize}
\bi $L_p$ surface area measure.
$S_p(K,\cdot)= n\widetilde{C}_{p,n}(K,B^n;\cdot)$.

\bi $L_p$ integral curvature.  $J_p(K,\cdot) =n \wt C_{p,0}(K^*,B^n;\cdot)$.

\bi Dual curvature measure. $\wt C_{0,q}(K,B^n;\cdot)$.
\end{itemize}

LYZ  \cite{LYZ18} showed
that the $L_p$ dual curvature measures are weakly continuous.

\begin{lemma}[LYZ \cite{LYZ18}]
\label{CpqcontQ}
For $p, q\in\R$ and $Q\in\mathcal{S}_o^n$, if $K_m \in \kno$ with $K_m \to  K\in \kno$, then
$\widetilde{V}_{q}(K_m,Q) \to \widetilde{V}_{q}(K,Q)$, and
$\widetilde{C}_{p,q}(K_m,Q;\cdot) \to \widetilde{C}_{p,q}(K,Q;\cdot)$ weakly.
\end{lemma}

\section{Classes of Convex Bodies with Group Symmetry}\label{sect-classes}

First, we describe constructions of closed subgroups of $O(n)$ that have no non-zero fixed points
and do not contain the negative identity $-I$ of $O(n)$.  Then, we construct convex bodies
invariant under such a subgroup $G$. It is important to note that these classes convex bodies are typically not origin-symmetric.
They provide geometric significance for the Monge-Amp\`ere equation \eqref{pde} with group symmetry
and are also of independent geometric interest.

For a closed subgroup $G$ of $O(n)$, recall the decomposition of $\rn$ into $G$-invariant
irreducible subspaces (see Fulton-Harris \cite{FuH04}), $\rn=\oplus_{i=1}^kV_i$, where $V_i$ are
pairwise orthogonal $G$-invariant irreducible subspaces of dimension at least 1.
The action of $G$ is irreducible on $V_i$,
if $V_i$ has no proper $G$-invariant subspaces with dimension at least $1$.
In addition, any $G$-invariant linear subspace $V\subset \R^n$ is the direct sum of some of the $V_i$'s.

If $G$ has a non-zero fixed point $x_1\in\R^n\backslash\{o\}$, that is $gx_1=x_1$, $x_1\neq 0$,
for any $g\in G$,
then $\R x_1$ is an invariant subspace and  $\R x_1=V_i$ for some $i\in\{1,\ldots,k\}$.
This implies that $G$ has no non-zero fixed point if ${\rm dim}\,V_i\geq 2$.

Denote by $O(V)$ the orthogonal group in the subspace $V\subset \rn$ and $I_V$ the identity
of $O(V)$.
We observe that the condition that $-I\not\in G$ can be ensured if $-I_{V_1}\not\in G|_{V_1}$, where
$G|_{V_1}$ is the subgroup of $O(V_1)$ formed by the restrictions of the elements of $G$ onto the $G$-invariant subspace $V_1$.

By the observations above, one can construct a closed subgroup $G$ of $O(n)$ that has no non-zero fixed
point and does not contain the negative identity $-I$.

\begin{const}[\bf Subgroup $G\subset O(n)$ with no non-zero fixed point and $-I\not\in G$]
\label{group-examples}	\

\begin{itemize}
\bi Consider the decomposition of pairwise orthogonal linear subspaces,
\[
\R^n=V_1\oplus \cdots \oplus V_k, \ \ \ {\rm dim}\,V_i\geq 2,\ i=1,\ldots,k.
\]

\bi Choose some closed subgroup $G_i\subset O(V_i)$ in such a way that each $G_i$'s action on $V_i$ is irreducible, and  $-I_{V_1}\not\in G_1$.

\bi Let
\[
G=G_1\oplus \cdots \oplus G_k.
\]

\bi Examples of $G_1$.  For simplicity, let $G_1=H$ and $V_1=\R^m$, $m\geq 2$. The following are
examples that the action of $H$ on $\R^m$ is irreducible and $-I_{\R^m}\not\in H$:
\begin{itemize}
\sir{i} The symmetry group of a regular simplex in $\R^m$ whose centroid is the origin. This group is isomorphic  to $S_{m+1}$, the group of permutations of $(m+1)$ elements.
		
\sir{ii} The group of orientation preserving symmetries (symmetries that are elements of ${\rm SO}(m)$) of  a regular simplex in $\R^m$ whose centroid is the origin. This group is isomorphic  to the  alternating group, the index $2$ subgroup $A_{m+1}\subset S_{m+1}$.
		
\sir{iii} If $m\geq 3$ is odd, then take the group of orientation preserving symmetries of the cube $[-1,1]^n$.
	
\sir{iv} If $m\geq 3$ is odd, then take any subgroup $H\subset {\rm SO}(m)$ acting irreducibly on $\R^m$.
	
\sir{v} For $m=2$ and an odd $\ell\geq 3$, take the group of orientation preserving symmetries (rotations) of a regular $\ell$-gon  whose centroid is the origin. This group is isomorphic  to the cyclic group $C_{\ell}$.
\end{itemize}
\end{itemize}
\end{const}

Naturally, if $m$ is odd, then (ii) and (iii) are examples for (iv). We note that for $m=3,5$, Zimmermann \cite{Zim12} lists all subgroups of $SO(m)$ acting irreducibly on $\R^m$ and discusses properties of subgroups of $SO(m)$ in higher dimensions. One notable example is the group of oriention-preserving symmetries of a regular dodecahedron, which is actually $A_5$.
	
The monograph Fulton-Harris \cite{FuH04} describes irreducible actions of many finite groups. As $A_d\subset S_d$ has index $2$,  for an irreducible action of $S_d$ on $\R^m$ via elements of $O(m)$, the restriction from $S_d$ to $A_d$ acts via elements of $SO(m)$
(see Fulton-Harris \cite{FuH04}, Section 5.1, for the construction of such representations and the description
of the resulting action of $A_d$ to be still irreducible).

\vspace{10pt}

Given a closed subgroup $G\subset O(n)$ that has no non-zero fixed point and does not contain $-I$,
we now construct $G$-invariant convex bodies that are not origin-symmetric. Since the centroid of a convex
body is equivariant with respect to linear transformations, the centroid of a $G$-invariant convex body
is necessarily $G$-invariant. However, $G$ has no non-zero fixed point, and thus the centroid of any such $G$-invariant
convex body is necessarily the origin.

\begin{const}[\bf Non-symmetric, centered, $G$-invariant convex bodies, I] \label{Invariant-Body-intersection}
 Let $G\subset O(n)$ be a finite subgroup that has no non-zero fixed point and does not contain $-I$,
 and let $C$ be a convex body in $\kno$ with unique minimal radius
 $\rho\lsub{C}(u_1) = \min_{u\in\sn} \rho\lsub{C}(u)$.
 Then for almost all $h\in O(n)$, the convex body  $K=\cap\{ghC: \,g\in G\}$ is $G$-invariant and not
 origin-symmetric.
\end{const}

\begin{proof} It is clear that $K$ is $G$-invariant.
Let $r=\rho\lsub{C}(u_1)$ and $z=ru_1$. Then
\[
rB^n \subset C, \ \  rB^n \cap \partial C =\{z\}, \ \text{ and  }\ -z \in \text{int} C.
\]
 As $-I\not\in G$, for each $g\in G\backslash \{I\}$, there exists an $(n-1)$-dimensional linear subspace $L_g\subset\R^n$ that contains all $x\in\R^n$ satisfying $gx=-x$. Since $G$ is finite,
for almost all $h\in O(n)$ with respect to the Haar measure of $O(n)$,
$hz\not\in \cup \{L_g: g\in G\backslash \{I\}\}$. Consequently, $-hz \not\in \cup\{ghz : g\in G\}$.
For such an $h$, let  $K=\cap\{ghC:\,g\in G\}$. We have
\[
rB^n \subset K, \ \ rB^n \cap \partial K = \cup\{ghz : g\in G\}, \ \text{ and }\  -hz \in \text{int} K.
\]
Thus, $hz \in \partial K$ and $-hz \in \text{int} K$. Therefore, $K$ is not origin-symmetric.
\end{proof}

In fact, for any convex body $C_0\subset\R^n$, $\partial C_0$ is twice differentiable in the Alexandrov sense for $\mathcal{H}^{n-1}$ a.e. points of $\partial C_0$. If $z_0\in \partial C_0$ is such a point, then there exists $r>0$ and $x_0\in C_0$ such that $x_0+rB^n\subset C_0$ and $(x_0+rB^n)\cap \partial C_0=\{z_0\}$, and hence one may choose $C=C_0-x_0$ in the construction above.

The next construction is, in some sense, dual to the one in Construction~\ref{Invariant-Body-intersection}.

\begin{const}[\bf Non-symmetric, centered, $G$-invariant convex bodies, II]	\label{Invariant-Body-DVcell}
 Let $G\subset O(n)$ be a finite subgroup that has no non-zero fixed point and does not contain $-I$,
 and let $C$ be a convex body in $\kno$ with unique maximal radius
 $\rho\lsub{C}(u_1) = \max_{u\in\sn} \rho\lsub{C}(u)$.
 Then for almost all $h\in O(n)$, the convex body  $K=\cap\{ghC: \,g\in G\}$ is $G$-invariant and not
 origin-symmetric.
\end{const}

\begin{proof}	
It is obvious that $K$ is $G$-invariant.
Let $R=\rho\lsub{C}(u_1)$ and $z=ru_1$. Then
\[
C\subset RB^n, \ \  RB^n \cap \partial C =\{z\}, \ \text{ and  }\ -z \not\in C.
\]
 As $-I\not\in G$, for each $g\in G\backslash \{I\}$, there exists an $(n-1)$-dimensional linear subspace $L_g\subset\R^n$ that contains all $x\in\R^n$ satisfying $gx=-x$. Since $G$ is finite,
for almost all $h\in O(n)$ with respect to the Haar measure of $O(n)$,
$hz\not\in \cup \{L_g: g\in G\backslash \{I\}\}$, and consequently, $-hz \not\in \cup\{ghz : g\in G\}$.
For such an $h$, let  $K=\cap\{ghC:\,g\in G\}$. We have
\[
K\subset RB^n, \ \ RB^n \cap \partial K = \cup\{ghz : g\in G\}, \ \text{ and }\  -hz \not\in K.
\]
Thus, $hz \in \partial K$ and $-hz \not\in K$. Therefore, $K$ is not origin-symmetric.
\end{proof}

For any convex body $C_0\subset\R^n$, let $x_0$ be the circumcenter of $C_0$; namely,  $C_0\subset x_0+R_0B^n$ where $R_0$ is the minimal radius of ball containing $C_0$. In this case, there exists $k\in\{1,\ldots,n\}$ and $y_0,y_1,\ldots,y_k\in C_0\cap \partial (x_0+R_0B^n)$ such that $x_0$ lies in the relative interior of ${\rm conv}\{y_0,\ldots,y_k\}$. We move $x_0$ into a position $x'\in {\rm conv}\{y_0,\ldots,y_k\}$ away from $y_0$, such that $x'\neq x_0$ and $x_0\in {\rm conv}\{x',y_0\}$. Hence, we may choose $C=C_0-x'$ and $R=\|y_0-x'\|$ in the construction above. If $C_0$ is a polytope, then $K$ is a polytope; however, if $C_0$ has $C^1$ boundary, then $K$ has $C^1$ boundary, as well.

In Constructions~\ref{Invariant-Body-intersection} and \ref{Invariant-Body-DVcell}, one does not have much control of the boundary of the resulting $G$-invariant convex body $K$. In
Construction~\ref{Invariant-Body-convexhull}, we shall construct $\partial K$ piece by piece. The idea is as follows:

For example, if $G$ is the group of orientation preserving symmetries of a regular $m$-gon $P$ centered at the origin for odd $m\geq 3$, then let $v_1,\ldots,v_m$ be the vertices of $P$ in this order, and let $\ell^+_j$ be the supporting half-plane containing $P$ such that the bounding line contains $v_j$ and is orthogonal to $v_j$. Then we can arbitrarily prescribe the part $\sigma_K$ between $v_1$ and $v_2$ of the boundary of a $G$-invariant convex body $K$ with $v_1,\ldots,v_m\in\partial K$, as long as the convex arc $\sigma_K$ satisfies $\sigma_K\subset \ell^+_1\cap \ell^+_2$. To ensure $K$ is not $o$-symmetric, we also make sure that $-v_{\frac{m+3}2}\not\in\sigma_K$.

\begin{const}[\bf Non-symmetric, centered, $G$-invariant convex bodies, III]	 \label{Invariant-Body-convexhull}
Let $G\subset O(n)$ be a finite subgroup that has no non-zero fixed point and does not contain $-I$.
For $g\in G\backslash \{I\}$, let $L_g$ and $\wt L_g$ be $(n-1)$-dimensional linear subspaces in $\rn$ 
that contain all solutions $x\in\R^n$ satisfying $gx=-x$ and $gx=x$ respectively. 
Fix a $\tilde{z}\in S^{n-1}$ such that $\tilde{z}\not\in L_g\cup \widetilde{L}_g$ for any $g\in G\backslash \{I\}$. Then $g\tilde{z}\neq -\tilde{z}$ and $g\tilde{z}\neq h\tilde{z}$ for any $g\in G$ and $h\in G\backslash \{g\}$.
Consider the Dirichlet-Voronoi cell of $\tilde{z}$ with respect to the orbit $\{g\tilde{z}\,:g\in G\}$; namely,
\begin{align*}
D=&\{x\in\R^n:\|x-\tilde{z}\|\leq \|x- g\tilde{z}\|\mbox{ for }g\in G \}\\ =&\{x\in\R^n:\langle g\tilde{z}-\tilde{z},x\rangle\leq 0\mbox{ for }g\in G \}.
\end{align*}
Observe that $D$ is a polyhedral cone that is the closure of the  fundamental domain for the action of $G$. Namely, for any $x\in\R^n$, there exists some $g\in G$ such that $x\in gD$. However,
${\rm int}(gD)\cap {\rm int}(hD)=\emptyset$ for $g,h\in G$ with $g\neq h$.

We now indicate a way to construct $G$-invariant convex bodies by prescribing their intersection with $D$ rather arbitrarily. Precisely, we construct a convex body $K_0=K\cap D$
such that $K=\cup\{gK_0:\,g\in G\}$ is a convex body, and hence $K$ is $G$-invariant with
 $\partial K_0\cap{\rm int}\,D\subset\partial K$.
 Let $\mathcal{F}_i$ be the (finite) set of $i$-dimensional faces of $D$, $i=1,\ldots,n$. In addition,
we choose representatives $F_{i,1},\dots,F_{i,k_i}\in  \mathcal{F}_i$, $k_i\geq 1$, for $i=2,\ldots,n$ such that any $F\in  \mathcal{F}_i$ can be written in the form $F=gF_{i,j}$ for some $g\in G$ and $F_{i,j}$, but $F_{i,j}=gF_{i,\ell}$ for
$g\in G$, and $1\leq \ell<j\leq k_i$ implies that $g=I$.
To actually construct $D\cap K$ for the $G$-invariant convex body $K$, first let $f\cap S^{n-1}\in\partial K$ for any $1$-dimensional face $f$ of $D$. Now we continue by induction on $i=2,\ldots,n$, while preserving the following  properties:
\begin{itemize}
\sir{i} If $F,gF\in \mathcal{F}_i$ for $i=1,\ldots,n-1$ and $g\in G$, then $K\cap(gF)=g(K\cap F)$.
\sir{ii} For $F_{i,j}$ and $x\in F_{i,j}\cap K$, if
\begin{itemize}
\bii either $M\in \mathcal{F}_{i-1}$
and $u\in S^{n-1}\cap {\rm lin}\,M$ is an exterior normal at a
$y\in \partial K\cap{\rm relint}\,M$;

\bii or $M=g F_{i,\ell}\subset D$ for $1\leq \ell<j$ and $g\in G$,
and $u\in S^{n-1}\cap {\rm lin}\,M$ is an exterior normal at a
$y\in \partial K\cap{\rm relint}\,M$,
\end{itemize}
then
$\langle x,u\rangle\leq \langle u,y\rangle$.
	
\end{itemize}

Following this procedure, we obtain a $G$-invariant convex body $K$. To ensure that $K$ is not $o$-symmetric, we act as follows. By perturbing $\tilde{z}$, we obtain a $z_0\in S^{n-1}\cap {\rm int}\,D$ such that $-z_0\neq g z_0$ for any $g\in G\backslash\{I\}$,  and $-z_0\in g_0({\rm int}\,D)$ for a $g_0\in G\backslash\{I\}$. In particular, $-z_0=g_0y_0$ for a $y_0\in {\rm int}\,D$ where $y_0\neq z_0$. Now in the final step of our construction (when $i=n$ above), we just ensure that $\varrho_K(y_0)\neq \varrho_K(z_0)$.
\end{const}

Given some classes of $G$-invariant convex bodies, one can construct $G$-invariant homogeneous
functions. Indeed, 
by transformation formulas \eqref{htrans} and \eqref{rhotrans}, one deduces that the support function
and the radial function of a $G$-invariant convex body are $G$-invariant. This 
provides examples of $G$-invariant homogeneous functions, and hence examples for the $G$-invariant
Monge-Amp\`ere type equation \eqref{pde}.

\section{Estimates for Dual Quermassintegrals}
\label{secDualIntVolumeBS}

One key step to solving the $L_p$ dual Minkowski problem is obtaining estimates for dual quermassintegrals, and
the following lemma provides essentially optimal ones.
Special cases, equivalents and more general versions were shown
by various authors, see for example \cite{JLW15, HaodiChen, HLYZ16, Zha18, BLYZ19}.
Chen \cite{HaodiChen} gave a proof for the exact estimates below and applied them
to prove a non-sharp Blaschke-Santal\'o-type inequality and its reverse inequality
(Theorem~\ref{qthIntrinsicBS}) for dual mixed volumes. To be self-contained,
we include an alternate proof.

\begin{lemma}
\label{qthIntrinsicBox}
Let $q>0$, $0<a_1\le \ldots \le a_n$, and
$R=[-a_1, a_1]\times \cdots \times [-a_n,a_n]$. There exists a constant $c_{n,q}$ such that
\[
\tilde V_q(R) \le c_{n,q} \begin{cases}
a_1\cdots a_i a_{i+1}^{q-i} &\text{if } i<q<i+1, i=0, \ldots, n-2, \\
a_1\cdots a_{n-1} a_{n}^{q-n+1} &\text{if } q> n-1, \\
a_1\cdots a_q \big(1+\log\frac{a_{q+1}}{a_q}\big) &\text{if } q=1, \ldots, n-1.
\end{cases}
\]
The reverse inequality is also true for a different constant $c_{n,q}$.
\end{lemma}

\begin{proof}
Let $R=D\times N \times J$, where $D=[-a_1, a_1]\times \cdots \times [-a_i,a_i]$,
$N=[-a_{i+1}, a_{i+1}]$, and $J=[-a_{i+2}, a_{i+2}]\times \cdots \times [-a_n,a_n]$.
Note that $D$ or $J$ does not appear if $i=0$ or $i=n-1$.

Let $\alpha_k$ be the surface area of the unit sphere in $\R^k$.
A simple calculation yields the following if $t>0$ and $k<p$,
\begin{equation}\label{jb}
\int_{z\in \R^k} (t+|z|)^{-p} \, dz = \frac{\alpha_k}{p-k} t^{k-p}.
\end{equation}

Let $x=(y,t,z)\in D\times N \times J$. Then $|x| \ge \frac12(|t|+|z|)$.
For $i<q<i+1, i=0, \ldots, n-2$, by \eqref{jb},
\begin{align*}
\tilde V_q(R) &= \frac qn \int_R |x|^{q-n} \, dx \le \frac qn 2^{n-q} V_i(D) \int_{N\times J} (|t|+|z|)^{q-n} dt dz \\
&\le \frac{q2^{n-q}\alpha_{n-i-1}}{n(i+1-q)} V_i(D) \int_{t\in N} |t|^{q-i-1}\, dt \\
&=\frac{q2^{n-q+1}\alpha_{n-i-1}}{n(i+1-q)(q-i)} V_i(D) a_{i+1}^{q-i}.
\end{align*}

When $n-1<q <n$,
\[
\tilde V_q(R) \le \frac qn \int_{D\times N} |t|^{q-n} \, dydt = \frac{q2^{n-q+1}}{n(q-n+1)} V_{n-1}(D) a_n^{q-n+1}.
\]

When $q\ge n$,
\begin{align*}
\tilde V_q(R) &\le \frac qn n^\frac{q-n}2 \int_{|x_i|\le a_i} (|x_1|^{q-n}+\cdots |x_n|^{q-n}) dx_1\cdots dx_n \\
&=\frac{2qn^{\frac{q-n}2-1}}{q-n+1} (a_1^{q-n+1}a_2\cdots a_n + \cdots a_1\cdots a_{n-1}a_n^{q-n+1}) \\
&\le \frac{2qn^{\frac{q-n}2}}{q-n+1}a_1\cdots a_{n-1}a_n^{q-n+1}.
\end{align*}

When $q=1, \ldots, n-1$, let $R=D\times N \times J$, where $D=[-a_1, a_1]\times \cdots \times [-a_{q-1},a_{q-1}]$,
$N=[-a_q, a_q]\times [-a_{q+1}, a_{q+1}]$, and $J=[-a_{q+2}, a_{q+2}]\times \cdots \times [-a_n,a_n]$.
Note that $D$ or $J$ does not appear if $i=1$ or $i=n-1$.
As shown above,  by \eqref{jb},
\begin{align*}
\tilde V_q(R) \le \frac{q2^{n-q}\alpha_{n-q-1}}{n} V_{q-1}(D) \int_{t\in N} |t|^{-1}\, dt,
\end{align*}
where $\alpha_0=1$ when $q=n-1$. Let $t=(t_1, t_2)$. Then
\begin{align*}
 \frac14\int_{t\in N} |t|^{-1}\, dt &\le \int_0^{a_{q+1}}\int_0^{a_q} \frac1{t_1+t_2}\, dt_1 dt_2 \\
 &=a_q\log\big(1+\frac{a_{q+1}}{a_q}\big) + a_{q+1}\log\big(1+\frac{a_{q}}{a_{q+1}}\big) \\
 &\le a_q\log\big(1+\frac{a_{q+1}}{a_q}\big) + a_q
   \le 2 a_q\big(1 + \log\frac{a_{q+1}}{a_q}\big).
\end{align*}
Thus,
\[
\tilde V_q(R) \le \frac{q2^{n-q+3}\alpha_{n-q-1}}{n} V_{q-1}(D)  a_q\big(1 + \log\frac{a_{q+1}}{a_q}\big).
\]

The reverse inequality is easier to show. Let $q>0$, $i=0, 1, \ldots, n-1$, and
\[R'=\{x=(x_1,\ldots,x_n) : |x_j|\le a_j \text{ if } j\le i, \frac12 a_{i+1} \le |x_j| \le a_{i+1}
\text{ if } j>i\}.\]
 Then for $x\in R'$,
$\frac12 a_{i+1} \le |x| \le n^\frac12 a_{i+1}$. Thus,
\[
\tilde V_q(R) \ge \tilde V_q(R') \ge \frac qn c^{q-n} a_{i+1}^{q-n} V_n(R')
=\frac qn c^{q-n} 2^i a_1\cdots a_i a_{i+1}^{q-i},
\]
where $c=n^\frac12$ if $q\le n$ and $c=\frac12$ if $q>n$. When $q=i$, it gives that
\begin{equation}\label{r1}
\tilde V_i(R) \ge \frac in i^{i-n} 2^i a_1\cdots a_i, \ \ i=1, \ldots, n-1.
\end{equation}
Let
\[
R'=\{x=(y,z)\in \R^i\times \R^{n-i} : |y_j|\le a_j \text{ if } j=1,\ldots, i, a_i\le |z|\le a_{i+1}\}.
\]
Then $|x|\le i^\frac12 (a_i + |z|)$. Thus,
\begin{align}
\tilde V_i(R) &\ge \tilde V_i(R') \ge \frac{i^{\frac{i-n}2+1}}{n}
 a_1\cdots a_i \int_{a_i\le |z|\le a_{i+1}} (a_i + |z|)^{i-n}\, dz \nonumber \\
&= \frac{i^{\frac{i-n}2+1}}{n} \alpha_{n-i} a_1\cdots a_i
\int_1^\frac{a_{i+1}}{a_i} \big(\frac1t+1\big)^{i-n} \frac1t \, dt \nonumber\\
&\ge \frac{i^{\frac{i-n}2+1}}{n} \alpha_{n-i} 2^{i-n} a_1\cdots a_i
\int_1^\frac{a_{i+1}}{a_i}\frac1t \, dt \nonumber\\
&=\frac{i^{\frac{i-n}2+1}}{n} \alpha_{n-i} 2^{i-n} a_1\cdots a_i  \log \frac{a_{i+1}}{a_i}. \label{r2}
\end{align}
Combining \eqref{r1} and \eqref{r2} proves the case of $q=1, \ldots, n-1$.
\end{proof}

In the rest of this section, we discuss some consequences of Lemma~\ref{qthIntrinsicBox}.
Recall that for a convex body $K\in\mathcal{K}^n_{o}$, its polar (cf. Schneider \cite{Sch14}) is the convex body $K^*\in\mathcal{K}^n_{o}$ defined as
$$
K^*=\{x\in\R^n:\,\langle x,y\rangle\leq 1\,\;\forall y\in K\}.
$$
It satisfies $(K^*)^*=K$, $(\lambda K)^*=\lambda^{-1}K^*$ for $\lambda>0$. Then if $E$ is a centroid, $E^*$ is a centered ellipsoid as well, and if $K\subset C$ for $K,C\in\mathcal{K}^n_{o}$, then $C^*\subset K^*$. In addition,
if $v\in S^{n-1}$, then
\begin{equation}
\label{SupportRadialPolar}
\rho\lsub{K^*}(v)=h_K(v)^{-1}.
\end{equation}
For example, if $\Gamma=\sum_{i=1}^n[-a_i,a_i]e_i$ for $a_1,\ldots,a_n>0$ and
orthonormal basis $e_1,\ldots,e_n$ of $\R^n$, then
\begin{equation}
\label{BoxPolar}
\Gamma^*={\rm conv}\left\{\frac{\pm e_1}{a_1},\ldots,\frac{\pm e_n}{a_n}\right\}\supset
\sum_{i=1}^n\left[\frac{-1}{na_i},\frac{1}{na_i}\right]e_i.
\end{equation}

One of the most well-known inequalities involving the polar body is the Blaschke-Santal\'o inequality,
\begin{equation}
\label{BScentered}
V(K)\cdot V(K^*)\leq \kappa_n^2
\end{equation}
for any centered convex body $K$ in $\R^n$  (with centroid at the origin). We note that $K$ is centered if and only if the so-called Santal\'o point
of $K^*$ is the origin, where the Santal\'o point $z_C\in {\rm int}\,C$  of a $C\in\mathcal{K}^n_{o}$ is the unique point such that
$V((C-z_C)^*)\leq V((C-z)^*)$ for any $z\in{\rm int}\,C$. Naturally, if $z_K=o$ for the  Santal\'o point of $K$, then \eqref{BScentered} holds, as well.

On the other hand, Kuperberg \cite{Kup08} proved the following Reverse Blaschke-Santal\'o inequality (see the arXiv version for the case of non-symmetric convex bodies): If $K\in\mathcal{K}^n_{o}$, then
\begin{equation}
\label{RBSKuperbergeq}
V(K)\cdot V(K^*)> \frac{\kappa_n^2}{4^n},
\end{equation}
and furthermore $V(K)\cdot V(K^*)>\frac{\pi^n}{n!}> \frac{\kappa_n^2}{2^n}$ provided $K=-K$.

If $K\in\mathcal{K}^n_{o}$ is centered, or its Santal\'o point is the origin (and hence $K^*$ is centered), then Kannan-Lov\'asz-Simonovits \cite{KLS95} proved the existence of a
 centered ellipsoid $E\subset\R^n$ such that
\begin{equation}
\label{EellipsoidnE}
E\subset K\subset nE.
\end{equation}
In particular, if $e_1,\ldots,e_n$ form the
orthonormal basis  of $\R^n$ corresponding to the principal directions of $E$,  and
$a_1,\ldots,a_n>0$ are the corresponding half axes of $E$, then
the rectangular box
$\Gamma=\sum_{i=1}^n[-a_1,a_1]e_i$ satisfies
\begin{equation}
\label{RboxBigR}
 n^{-1}\Gamma\subset K\subset n\Gamma,
\end{equation}
and hence  $\widetilde{\Gamma}=\sum_{i=1}^n\left[\frac{-1}{a_i},\frac{1}{a_i}\right]e_i$ satisfies
(cf. \eqref{BoxPolar})
\begin{equation}
\label{RboxPolarBigR}
n^{-2}\widetilde{\Gamma}\subset K^*\subset n\widetilde{\Gamma}.
\end{equation}
The first consequence of Lemma~\ref{qthIntrinsicBox} is a simple estimate on the inradius, which follows from \eqref{RboxBigR}, Lemma~\ref{qthIntrinsicBox}, and the $q$-homogeneity of the $q$-th dual quermassintegral.

\begin{lemma}
\label{inradiusVq}
For $q>0$, $R,c>1$, and $n\geq 2$, there exists $\xi=\xi(n,q,R,c)>0$ such that if $K\subset RB^n$ is a centered convex body,
$Q\in\mathcal{S}_o^n$ with $c^{-1}B^n\subset Q\subset cB^n$,   and $\widetilde{V}_q(K,Q)\geq t$ for $t>0$, then  $\xi  t^{\frac1q}B^n\subset K$.
\end{lemma}

For $q>0$, $q^*$ is defined by \eqref{q*}. Otherwise, $r\in (0, q^*]$ if and only if
\begin{equation}
\label{q*rsup}
\frac{n-1}q+\frac1r\geq 1\mbox{ \ and \ }\frac{n-1}r+\frac1q\geq 1;
\end{equation}
that is, $q^*$ is the supremum of all $r>0$ satisfying \eqref{q*rsup}.

The following Blaschke-Santal\'o-type inequality for dual mixed volumes was shown
by Chen \cite{HaodiChen} as a consequence of Lemma~\ref{qthIntrinsicBox}.

\begin{theo}[H. Chen]
\label{qthIntrinsicBS}
For $n\geq 2$, $c>1$, $Q_1,Q_2\in\mathcal{S}_o^n$ with $c^{-1}B^n\subset Q_1,Q_2\subset cB^n$, and  $q,r>0$ such that $r\leq q^*$ (or equivalently, $q\leq r^*$), there exists $\theta(n,q,r,c)>1$ depending on $n,q,r,c$ such that if $K\subset\R^n$  is a centered convex body, then
\begin{equation}\label{BSi}
\theta(n,q,r,c)^{-1}\leq \widetilde{V}_q(K,Q_1)^{\frac1q}\cdot \widetilde{V}_r(K^*,Q_2)^{\frac1r}\leq \theta(n,q,r,c).
\end{equation}
\end{theo}
\noindent{\bf Remark.} If $Q_2=B^n$, then \eqref{SupportRadialPolar} and \eqref{BSi} yield
\begin{equation}
\label{qthIntrinsicBS-eq}
\theta(n,q,r,c)^{-1}\leq \widetilde{V}_q(K,Q_1)^{\frac1q}\cdot \left(\frac1n\int_{S^{n-1}}h_K(u)^{-r}\,du\right)^{\frac1r} \leq \theta(n,q,r,c).
\end{equation}

\section{Existence of Solutions to the $L_p$ Dual Minkowski Problem}

First, we verify a simple statement about invariant measures.

\begin{lemma}
\label{Invariant-Measure}
Let $G$ be a closed subgroup of $O(n)$, $n\geq 2$. If $\mu_1$ and $\mu_2$
are $G$-invariant finite Borel measures on $S^{n-1}$, then $\mu_1=\mu_2$ if and only if $\int_{S^{n-1}}\varphi\,d\mu_1=\int_{S^{n-1}}\varphi\,d\mu_2$ for any $G$-invariant
continuous function $\varphi:S^{n-1}\to\R$.
\end{lemma}

\begin{proof}
The only if direction is obvious. For the if direction, we assume that
\begin{equation}\label{inv-meas-cond}
\int_{S^{n-1}}\varphi\,d\mu_1=\int_{S^{n-1}}\varphi\,d\mu_2
\end{equation}
holds for any $G$-invariant continuous function $\varphi:S^{n-1}\to\R$,
and we want to show that $\mu_1=\mu_2$.
Or equivalently, we shall show that
\begin{equation}
\label{invariant-measure-continuous}
\int_{S^{n-1}}\psi\,d\mu_1=\int_{S^{n-1}}\psi\,d\mu_2
\end{equation}
 for any continuous function $\psi:S^{n-1}\to\R$ (without assuming $G$-invariance). Since $G$ is a compact group, there exists a $G$-invariant probability measure (the Haar measure) $\nu$ on $G$.
 Let
 $$
 \varphi(u)=\int_G\psi(gu)\,d\nu(g), \ \ u\in S^{n-1}.
 $$
 It follows that $\varphi$ is a $G$-invariant continuous function. By Fubini's theorem, a change of variables, and the $G$-invariance of $\mu_i$, we have that
 \[
\begin{aligned}\int_{S^{n-1}}\varphi\text{ }d\mu_{i} & =\int_{S^{n-1}}\int_{G}\psi(gu)\text{ }d\nu(g)d\mu_{i}(u)\\
 & =\int_{G}\int_{S^{n-1}}\psi(gu)\text{ }d\mu_{i}(u)d\nu(g)\\
 & =\int_{G}\int_{S^{n-1}}\psi(v)\text{ }d\mu_{i}(g^{-1}v)d\nu(g)\\
 & =\int_{G}\int_{S^{n-1}}\psi(v)\text{ }d\mu_{i}(v)d\nu(g)\\
 & =\int_{S^{n-1}}\psi(v)\text{ }d\mu_{i}(v).
\end{aligned}
\]
By this and \eqref{inv-meas-cond}, we conclude \eqref{invariant-measure-continuous}.
\end{proof}

For $C\in\mathcal{K}_o^n$, we consider the functional
$$
\Phi(C)=
 \frac1p\log\int_{S^{n-1}}h_{C}^p\,d\mu-\frac1q\log\widetilde{V}_q(C,Q),
$$
which is scale-invariant. Namely, it satisfies
\begin{equation}
\label{entropy-rescale-invariant}
\Phi(\lambda\,C)=\Phi(C) \mbox{ \ for $\lambda>0$}.
\end{equation}
In addition, the functional is continuous; namely, Lemma~\ref{CpqcontQ} yields that if $C_m\in\mathcal{K}_o^n$ tends to $C\in\mathcal{K}_o^n$, then
\begin{equation}
\label{entropy-continuous}
\lim_{m\to\infty}\Phi(C_m)=\Phi(C).
\end{equation}

Let $\mathcal{C}$ be the set of $G$-invariant convex bodies $C\subset\R^n$ with $\widetilde{V}_q(C,Q)=1$.
 Any $C\in\mathcal{C}$ is centered, i.e. the centroid of $C$ is at the origin since $C$ is $G$-invariant, $G$ has no non-zero fixed point, and the centroid is equivariant with respect to linear transformations. Thus, for $C\in\mathcal{C}$, we have
$$
\Phi(C)=
 \frac1p\log\int_{S^{n-1}}h_{C}^p\,d\mu.
$$

The following lemma gives the existence of solutions to the minimization problem of $\Phi(C)$.

\begin{lemma}\label{compactness}
Let $q>0$, $-q^*<p<0$, $G$ a closed subgroup of $O(n)$
without a non-zero fixed point, and
$Q$ a $G$-invariant star body in $\rn$.
Let $\mu$ be a non-trivial, $G$-invariant, finite Borel
measure on $\sn$ with a density function $f\in L^s(\sn)$, where
$s>1$ if $q\leq 1$, and $s=\frac{1}{1+p/q^*}>1$ if $q>1$. Then
\begin{itemize}
\ir{1} If $C_m\in\mathcal{C}$ satisfies  $\lim_{m\to\infty}{\rm diam}(C_m)=\infty$, then
\begin{equation}\label{Cmdiaminfinity}
\lim_{m\to\infty}\Phi(C_m)=\infty.
\end{equation}

\ir{2} There exists $\widetilde{K}\in\mathcal{C}$, such that
\begin{equation}
\label{tildeCminimizes}
\Phi(\widetilde{K})=\min\{\Phi(C):\,C\in\mathcal{C}\}.
\end{equation}
\end{itemize}
\end{lemma}

\begin{proof}
(1) It is equivalent to prove that
\begin{equation}
\label{Cmdiaminfinity0}
\lim_{m\to\infty}\int_{S^{n-1}}h_{C_m}^{p}\,d\mu=0.
\end{equation}
We choose $c>1$ so that $c^{-1}B^n\subset Q\subset cB^n$.
Let $\tilde{q}=q^*$ if $q>1$, and let $\tilde{s}=\frac{\tilde{q}}{\tilde{q}+p}$ in this case; moreover,
let $\tilde{q}>|p|$ be such that $\tilde{s}=\frac{\tilde{q}}{\tilde{q}+p}\leq s$ if $q\leq 1$, and hence
$f\in L^{\tilde{s}}(S^{n-1})$ for any $q>0$.

Let $R_m=\max\{\|x\|:\,x\in C_m\}$. Then by assumption, $\lim_{m\to\infty}R_m=\infty$.
We choose $v_m\in S^{n-1}$ such that $R_mv_m\in C_m$, observe that $R_mv\in C_m$ for $v\in Gv_m$,
and consider the set
$$
\Sigma_m=\left\{u\in S^{n-1}:\,\langle u,v\rangle\leq R_m^{-1/2}\,\;\forall v\in Gv_m\right\}.
$$
To prove \eqref{Cmdiaminfinity0}, it suffices to verify that for any $\varepsilon\in(0,1)$,
if $m$ is large, then we have the following:
\begin{align}
\label{mEpsilonOutSigma}
\int_{S^{n-1}\backslash \Sigma_m}h_{C_m}^{p}\,d\mu&\leq \varepsilon\\
\label{mEpsilonInSigma}
\int_{\Sigma_m}h_{C_m}^{p}\,d\mu&\leq n^{\frac{|p|}{\tilde{q}}}\theta(n,q,\tilde{q},c)^{|p|}\cdot \varepsilon
 \ \  \ \ \mbox{ if $\Sigma_m$ is non-empty}.
\end{align}
For \eqref{mEpsilonOutSigma}, if $u\in S^{n-1}\backslash \Sigma_m$ for large $m$, then there exists $v\in Gv_m$ with
$\langle u,v\rangle> R_m^{-1/2}$, and hence $h_{C_m}(u)\geq \langle u,R_mv\rangle> R_m^{1/2}$. In particular,
$$
\int_{S^{n-1}\backslash \Sigma_m}h_{C_m}^{p}\,d\mu\leq R_m^{-|p|/2}\mu(S^{n-1}),
$$
verifying \eqref{mEpsilonOutSigma}.

Turning to \eqref{mEpsilonInSigma}, for any $w\in S^{n-1}$ and $\delta\in[0,1)$, let
$$
\Xi_{w,\delta}=\left\{x\in\R^n:\,\langle x,v\rangle\leq \delta\,\;\forall v\in Gw\right\}.
$$
We observe that $\Sigma_m=\Xi_{u_m,\delta}\cap S^{n-1}$ for $\delta=R_m^{-1/2}$.
Since $G$ has no non-zero fixed point, if $w\in S^{n-1}$, then $o\in{\rm conv}\,Gw$. Hence, $\Xi_{w,0}$ is contained a linear $(n-1)$-plane. Since $f\in L^{\tilde{s}}(S^{n-1})$, there exists $\tilde{\delta}>0$ such that if $w\in S^{n-1}$, then either
$\Xi_{w,\tilde{\delta}}\cap S^{n-1}=\emptyset$, or
\begin{equation}
\label{mEpsilonInSigmaXi}
 \left(\int_{\Xi_{w,\tilde{\delta}}\cap S^{n-1}}f(u)^{\frac{\tilde{q}}{\tilde{q}+p}}\,du\right)^{\frac{\tilde{q}+p}{\tilde{q}}}\leq \varepsilon.
\end{equation}
Note that $\Sigma_m\subset \Xi_{u_m,\tilde{\delta}}\cap S^{n-1}$ for large $m$. Hence, if
$\Sigma_m\neq\emptyset$, then we deduce from Chen's inequality \eqref{qthIntrinsicBS-eq}, $\widetilde{V}_q(C,Q)=1$,
   H\"older's inequality, and \eqref{mEpsilonInSigmaXi} that
\begin{align*}
 \int_{\Sigma_m}h_C^{p}\,d\mu&=\int_{\Sigma_m}h_C(u)^{-|p|}f(u)\,du\\
&\leq
\left(\int_{\Sigma_m}h_C(u)^{-\tilde{q}}\,du\right)^{\frac{|p|}{\tilde{q}}}
 \left(\int_{\Sigma_m}f(u)^{\frac{\tilde{q}}{\tilde{q}+p}}\,du\right)^{\frac{\tilde{q}+p}{\tilde{q}}}\\
& \leq n^{\frac{|p|}{\tilde{q}}}\theta(n,q,\tilde{q},c)^{|p|}\cdot\varepsilon.
\end{align*}
Therefore, we deduce \eqref{mEpsilonInSigma} as well, and in turn
\eqref{Cmdiaminfinity0} and \eqref{Cmdiaminfinity}.

(2)
Let $C_m\in\mathcal{C}$ be a minimizing sequence such that
$$
\lim_{m\to\infty}\Phi(C_m)=\inf\{\Phi(C):\,C\in\mathcal{C}\}.
$$
It follows from \eqref{Cmdiaminfinity} in Lemma \ref{compactness} that the minimizing sequence is bounded, i.e.,
there exists some finite $R>1$ such that
\begin{equation}
\label{CminRBn}
C_m\subset RB^n\mbox{ \ for all $C_m$}.
\end{equation}
Since $\widetilde{V}_q(C_m,Q)=1$, we deduce from  Lemma~\ref{inradiusVq} and \eqref{CminRBn}
the existence of a $\xi>0$ such that
\begin{equation}
\label{xiBninCm}
\xi\,B^n\subset C_m\mbox{ \ for each $C_m$}.
\end{equation}
Combining the Blaschke Selection Theorem with \eqref{CminRBn} and \eqref{xiBninCm} yields
that there exists a subsequence $\{C_{m'}\}\subset\{C_m\}$ tending to a centered convex body
$\widetilde{K}$. We conclude from
\eqref{entropy-continuous} and the continuity of dual intrinsic volume (cf. Lemma~\ref{CpqcontQ}), that
$\widetilde{K}\in\mathcal{C}$, and hence it satisfies \eqref{tildeCminimizes}.\\
\end{proof}

For reader's convenience, let us recall  Theorem~\ref{pneqqpos-abs-cont} as Theorem~\ref{pneqqpos-abs-cont0}.

\begin{theo}\label{pneqqpos-abs-cont0}
Let $q>0$, $-q^*<p<0$, $G$ a closed subgroup of $O(n)$
without a non-zero fixed point, and
$Q$ a $G$-invariant star body in $\rn$.
Let $\mu$ be a non-trivial, $G$-invariant, finite Borel
measure on $\sn$ with a density function
$f\in L^s(\sn)$, where $s>1$ if $q\leq 1$, and $s=\frac{1}{1+p/q^*}>1$ if $q>1$.
Then there exists a $G$-invariant convex body $K$ in $\rn$ so that it solves
the measure equation $\widetilde{C}_{p,q}(K,Q;\cdot)=\mu$.
\end{theo}

\begin{proof}
 We will show that the minimizer $\wt K$ of the functional $\Phi$ in Lemma \ref{compactness}
  corresponds
 to a solution to the $L^p$ dual Minkowski problem. In particular, we claim that
 for $\lambda=\int_{S^{n-1}}h_{\widetilde{K}}^p\,d\mu$, we have
\begin{equation}
\label{Step3-claim}
\mu=\lambda\, \widetilde{C}_q(\widetilde{K},Q;\cdot).
\end{equation}
According to Lemma~\ref{Invariant-Measure}, \eqref{Step3-claim} is equivalent to
\begin{equation}
\label{Step3-claim0}
\int_{S^{n-1}}\varphi\, h_{\widetilde{K}}^{p-1}\,d\mu=
\lambda\, \int_{S^{n-1}}\frac{\varphi}{h_{\widetilde{K}}}\,d\widetilde{C}_q(\widetilde{K},Q;\cdot)
\end{equation}
for any continuous $G$-invariant function $\varphi:\,S^{n-1}\to\R$.
For $t\geq 0$, we consider the Wulff shape
$$
K_t=\{x\in\R^n:\langle x,u\rangle\leq h_{\widetilde{K}}(u)+t\varphi(u)\;\forall u\in S^{n-1}\}.
$$
The Wulff shape is $G$-invariant, i.e. $K_t\in \mathcal C$, because for $g\in G$
\begin{align*}
gK_t &= \{gx \in \rn : \langle x, u\rangle \le h_{\wt K}(u) + t \varphi(u), u\in \sn \} \\
&=  \{y \in \rn : \langle g^{-1} y, u\rangle \le h_{\wt K}(u) + t \varphi(u), u\in \sn \} \\
&=  \{y \in \rn : \langle y, gu\rangle \le h_{\wt K}(gu) + t \varphi(gu), u\in \sn \} \\
&=  \{y \in \rn : \langle y, v\rangle \le h_{\wt K}(v) + t \varphi(v), v\in \sn \} =K_t.
\end{align*}

Note that $K_0=\widetilde{K}$, and so the variational formula \eqref{dualAlexandrov} yields that
\begin{equation}
\label{Step3-problem-K-vol}
\left.\frac{d }{dt}\,\widetilde{V}_q(K_t,Q)\right|_{t=0}=q\int_{S^{n-1}}\frac{\varphi}{h_{\widetilde{K}}}\,d \widetilde{C}_q(\widetilde{K},Q;\cdot).
\end{equation}
If $|t|$ is small, then we consider the differentiable function
$$
f(t)=\frac1p\log\int_{S^{n-1}}(h_{\widetilde{K}}+t \varphi)^p\,d\mu-\frac1q\log \widetilde{V}_q(K_t,Q).
$$
From the fact that
$\widetilde{V}_q(K_t,Q)^{\frac{-1}q}\cdot K_t\in\mathcal{C}$,
 $h_{K_t}\leq h_K(u)+t\varphi(u)$,
\eqref {entropy-rescale-invariant},
 and  \eqref{tildeCminimizes}, one obtains the following inequality

$$
f(t)\geq
\Phi(K_t)
=\Phi\left(\widetilde{V}_q(K_t,Q)^{\frac{-1}q}\cdot K_t\right)
\geq f(0).
$$
In particular, we have that $f$ has a minimum at $t=0$, and hence \eqref{Step3-problem-K-vol} with
$\widetilde{V}_q(K_0,Q)=1$ imply that
$$
0=f'(0)=\frac1{\lambda}\int_{S^{n-1}} \varphi\, h_K^{p-1}\,d\mu-\int_{S^{n-1}}\frac{\varphi}{h_{\widetilde{K}}}\,d \widetilde{C}_q(\widetilde{K},Q;\cdot),
$$
proving \eqref{Step3-claim0}, and in turn \eqref{Step3-claim}.\\

Finally, \eqref{Step3-claim} implies that
$\mu=\widetilde{C}_q(\widetilde{K},Q;\cdot)$ for $K=\lambda^{\frac1{q-p}}\widetilde{K}$, as desired.
\end{proof}

\end{document}